\documentclass{amsart}
\usepackage{empheq}

\usepackage{fullpage,url,amssymb,enumerate,colonequals}

\usepackage{mathrsfs} % for \mathscr (script letters)
\usepackage{MnSymbol}
\usepackage{extarrows}
\usepackage{lscape}
\usepackage[all,cmtip]{xy}
\usepackage[OT2,T1]{fontenc}
\usepackage{color}

\usepackage[
        colorlinks, citecolor=green,
        backref,
        pdfauthor={Imin Chen, Gleb Glebov, Ritesh Goenka},
]{hyperref}

\usepackage{float}
\restylefloat{table}

\usepackage[all]{xy}

\theoremstyle{definition}

\newtheorem{theorem}{Theorem}[section]

\newtheorem{definition}[theorem]{Definition}

\newtheorem{problem}[theorem]{Problem}

\newtheorem{lemma}[theorem]{Lemma}

\newcommand{\C}{\mathbb{C}}
\def\F21{\vphantom{\quad}_2 F_1}
\begin{document}

\title{Reflections on Euler's reflection formula and an additive analogue of Legendre's duplication formula}

\author{Ritesh Goenka and Gopala Krishna Srinivasan}

\date{\today}



\address{Ritesh Goenka\\
Department of Computer Science and Engineering\\
Indian Institute of Technology Bombay\\
Mumbai\\
Maharashtra\\
INDIA}

\email{160050047@iitb.ac.in}

\address{Gopala Krishna Srinivasan\\
Department of Mathematics\\
Indian Institute of Technology Bombay\\
Mumbai\\
Maharashtra\\
INDIA}

\email{gopal@math.iitb.ac.in}

\begin{abstract}
    In this note, we look at some of the less explored aspects of the gamma function. We provide a new proof of Euler's reflection formula and discuss its significance in the theory of special functions. We also discuss a result of Landau concerning the determination of values of the gamma function using functional identities. We show that his result is sharp and extend it to complex arguments. In 1848, Oskar Schl\"omilch gave an interesting additive analogue of the duplication formula. We prove a generalized version of this formula using the theory of hypergeometric functions.
\end{abstract}

\maketitle

\section{Introduction}

The study of special functions reveals a great deal of form, structure and symmetry often related to Lie groups of symmetries of differential equations. While its connection to number theory is at least three hundred years old\footnotemark\footnotetext{For instance, Euler's factorization of $\zeta(z)$.}, in the last many decades surprising connections with combinatorics and Hopf Algebras have been unearthed. Among the special functions, the gamma function occupies a rather distinguished position. In the present note, we discuss many less known properties of this function emphasizing the form and structure alluded above with some new proofs. Parts of the note deal with some aspects that seem to have fallen into oblivion whose origins can be traced back to the classic volumes of A. M. Legendre and Oskar Schl\"omilch. In fact, a problem proposed by Legendre was resolved by E. Landau, which we present briefly for the benefit of the readers. We also show that Landau's result is complete and extend it to values of the gamma function for complex arguments. Further, we substantially generalize an additive analogue of the duplication formula given by Schl\"omilch and provide a proof for the same towards the end of this note.

The gamma function first appeared in analysis nearly 300 years ago in a letter written by Euler to Goldbach in 1729. In the notation of A. M. Legendre, the function $\Gamma(z)$ is defined as 
\begin{equation}
\label{eqn:def}
    \Gamma(z) = \int_0^{\infty} t^{z-1}e^{-t} dt,\quad \mbox{Re}\;z > 0.
\end{equation}
It is easy to show using integral by parts in the above definition that $\Gamma$ satisfies the basic functional relation 
\begin{equation}
\label{eqn:fr}
    \Gamma(z+1) = z \Gamma(z), \quad \mbox{Re}\;z > 0.
\end{equation}
The integral in the RHS of \eqref{eqn:def} defines a holomorphic function in the right half-plane $R = \{z \in \C: \mbox{Re}\;z > 0\}$. By virtue of \eqref{eqn:fr}, $\Gamma$ continues analytically as a meromorphic function on the complex plane with simple poles at $0, -1, -2, \dots$, and in particular is devoid of any essential singularities in $\C$. We refer to \cite{srinivasan} for historical details and references to original sources. It is convenient to have at our disposal the closely related beta function (so named by Binet), which is defined as
\begin{equation*}
    \mathrm{B}(z, w) = \int_0^1 t^{z-1}(1-t)^{w-1} dt, \quad \mbox{Re}\;z > 0,\;\mbox{Re}\;w > 0. 
\end{equation*}
It is possible to obtain many other integral representations for the beta function by applying simple variable transformations in the above definition. We obtain the following useful representation by setting $s = t/(1-t)$ in the above definition.
\begin{equation}
\label{eqn:beta}
    \mathrm{B}(z, w) = \int_0^\infty \frac{s^{z-1}}{(1+s)^{z+w}} ds, \quad \mbox{Re}\;z > 0,\;\mbox{Re}\;w > 0.
\end{equation}
The two functions, beta and gamma, are connected via the famous beta-gamma relation of Euler given by
\begin{equation}
\label{eqn:bg}
    \Gamma(z)\Gamma(w) = \mathrm{B}(z, w)\Gamma(z+w),  \quad \mbox{Re}\;z > 0,\;\mbox{Re}\;w > 0.
\end{equation}
The reader may note with interest that the beta-gamma relation above bears close resemblance to a corresponding formula relating the character sums of Gauss and Jacobi in algebraic number theory \cite[p. 55]{ono}.

The reflection formula, discovered by Euler, given by
\begin{equation}
\label{eqn:erf}
    \Gamma(z)\Gamma(1-z) = \frac{\pi}{\sin \pi z},
\end{equation}
expresses the gamma function as ``one half of the sine function'' in a multiplicative sense. One can thus expect factorizations of the sine function to have gamma analogues and a notable case is the famous duplication formula discovered by A. M. Legendre in 1809. The function $\varphi(z) = \sin \pi z$ factorizes as 
\begin{equation}
\label{eqn:sine-fact}
    \varphi(z) = 2 \varphi\left(\frac{z}{2}\right) \varphi \left(\frac{z}{2} + \frac{1}{2}\right),
\end{equation}
with the corresponding gamma analogue being
\begin{equation}
\label{eqn:ldf}
    \sqrt{\pi}\Gamma(z) = 2^{z-1}\Gamma\left(\frac{z}{2}\right) \Gamma\left(\frac{z}{2} + \frac{1}{2}\right).
\end{equation}
The analogy between \eqref{eqn:sine-fact} and \eqref{eqn:ldf} is striking and even more so is the analogy between the submultiple angle formula 
\begin{equation*}
    \varphi(z) = 2^{k-1}\varphi\left(\frac{z}{k}\right) \varphi \left(\frac{z}{k} + \frac{1}{k}\right)\dots\varphi\left(\frac{z}{k} + \frac{k-1}{k}\right),  
\end{equation*}
and its gamma analogue (stated below) due to Gauss (1812).
\begin{equation}
\label{eqn:gmf}
    (2\pi)^{\frac{n-1}{2}} n^{\frac{1}{2} - z} \Gamma(z) = \prod_{j=0}^{n-1} \Gamma\left(\frac{z}{n} + \frac{j}{n}\right).
\end{equation}

\section{Proof of the reflection formula}

Numerous proofs of the reflection formula are available in the literature and proofs due to Dirichlet, Dedekind and Gauss are discussed in \cite{srinivasan} in addition to a new proof based on the additive approach to gamma function with the initial value problem 
\begin{equation*}
    \frac{d^2}{dx^2} (\log y) = \sum_{n = 1}^{\infty} \frac{1}{(x+n)^2}, \quad y(1) = 1, y^{\prime}(1) = \Gamma^{\prime}(1),
\end{equation*}
as a point of departure. Complete details of Dedekind's proof are available in \cite{srinivasan2}. Here, we shall discuss a proof with the flavour of partial differential equations. Although the argument given here has some features in common with the proof in \cite[Theorem 3.3]{srinivasan}, there are some essential differences which make this proof somewhat didactic. The crucial step in the argument is the non-vanishing of the gamma function. We show that the gamma function has no zeros in the complex plane, using the duplication formula for which we give a well-known elementary proof for completeness.

\begin{theorem}[Legendre's duplication formula]
    For $z$ in the right half-plane, \eqref{eqn:ldf} holds.
\end{theorem}

\begin{proof}
    Let $x$ be a positive real number. On setting $t = \sin^2u$ in the definition for $\mathrm{B}(x, x)$, we obtain
    \begin{equation*}
        2^{2x-1} \mathrm{B}(x, x) = 2 \int_0^{\pi/2} \sin^{2x-1}2u du = \int_0^{\pi} \sin^{2x-1}u du = 2 \int_0^{\pi/2} \sin^{2x-1}u du = \mathrm{B}(x, 1/2).
    \end{equation*}
    Appealing to the beta-gamma relation, we get the stated result for positive real values of $x$. The result follows for all $z$ in the right half-plane via analytic continuation.
\end{proof}

\begin{lemma}
    The function $\Gamma(z)$ has no zeros in $\C$.
\end{lemma}

\begin{proof}
    Suppose $z_0 \in \mathbb{C}$ is a zero of the gamma function. Since $\Gamma(z_0 + n)$ must also vanish for every natural number $n$, we may as well assume that $\textrm{Re}(z_0) > 0$. Now, \eqref{eqn:ldf} implies
    \begin{equation*}
        \Gamma\left(\frac{z_0}{2}\right) = 0,\; \textrm{or} \; \Gamma\left(\frac{z_0}{2} + \frac{1}{2}\right) = 0.
    \end{equation*}
    Define $z_1 = \frac{z_0}{2} + \frac{\epsilon_1}{2}$, where $\epsilon_1 = 0$ and $\epsilon = 1$ in the former and latter case respectively. Similarly, we obtain another zero $z_2 = \frac{z_1}{2} + \frac{\epsilon_2}{2}$ with $\epsilon_2 = 0$ or $1$. Continuing in this manner, we obtain a sequence of zeros $\{z_n\}_{n \in \mathbb{N} \cup \{0\}}$ satisfying the recursive relation
    \begin{equation*}
        z_n = \frac{z_{n-1}}{2} + \frac{\epsilon_n}{2},
    \end{equation*}
    with $\epsilon_n = 0$ or $1$, $\forall n \in \mathbb{N}$. Repeated application of the above relation yields
    \begin{equation}
    \label{eqn:exp}
        z_n = \frac{z_0}{2^n} + \sum_{j=1}^n \frac{\epsilon_j}{2^{n+1-j}},\;\forall n \in \mathbb{N}.
    \end{equation}
    We claim that if $j \ne k$, then $z_j \ne z_k$. Assume to the contrary that $z_j = z_k$ for some $j \ne k$. We may assume without loss of generality that $j < k$. Then, from \eqref{eqn:exp}, we have
    \begin{equation*}
        \textrm{Im}\left[\left(\frac{1}{2^j} - \frac{1}{2^k}\right) z_0\right] = 0,
    \end{equation*}
    which is a contradiction since the Gamma function has no positive real zeros (as can be easily seen from \eqref{eqn:def}). The sequence $\{z_n\}$ is evidently bounded since $|z_n| \le |z_0| + 1, \forall n \in \mathbb{N}$ from \eqref{eqn:exp}. Hence, it must contain a subsequence which converges to a point $w$. The point $w$ must be an essential singularity, but this is a contradiction since the only singularities of the gamma function are poles.
\end{proof}

\begin{lemma}[Liouville's theorem for harmonic functions]
\label{liouville}
    Assume that $h: \mathbb R^n \longrightarrow \mathbb R$ is a harmonic function satisfying the condition
    \begin{equation}
    \label{eqn:est}
        |h(x)| \leq |P(x)|,\;\forall x \in \mathbb{R}^n,
    \end{equation}
    where $P(x)$ is a polynomial of degree $k$. Then, $h(x)$ is itself a polynomial of degree at most $k$.
\end{lemma}

\begin{proof}
    The proof is certainly folk-lore. The estimate \eqref{eqn:est} implies that the function $h(x)$ defines a tempered distribution \cite{strichartz}, and so taking the Fourier transform of $\Delta h = 0$, we conclude
    \begin{equation*}
        |\xi|^2 \widehat{h}(\xi) = 0.
    \end{equation*}
    Thus, $\widehat{h}$ is a distribution with point support at the origin, and therefore is a linear combination of Dirac delta and finitely many of its derivatives \cite[p. 80]{strichartz}, i.e.
    \begin{equation*}
        \widehat{h} = \sum_{|\alpha| \leq k} c_{\alpha} \delta_0^{(\alpha)}.
    \end{equation*}
    Thus, $h(x)$ itself must be a polynomial and the estimate \eqref{eqn:est} forces the degree of $h$ to be at most $k$.
\end{proof}

We are now ready to prove the reflection formula. Since this part of the argument is identical to that in \cite[Theorem 3.3]{srinivasan}, the proof is kept to a bare minimum.

\begin{theorem}[Euler's reflection formula]
    For $z \in \mathbb{C}\setminus\mathbb{Z}$, \eqref{eqn:erf} holds.
\end{theorem} 

\begin{proof}
    Observe that the function $F$ defined by
    \begin{equation*}
        F(z) = z\Gamma(z)\Gamma(1-z)\frac{\sin \pi z}{z},
    \end{equation*}
    is an entire function devoid of zeros and has period one. Hence, $F(z) = \exp G(z)$ for some one-periodic entire function $G(z)$, and so
    \begin{equation*}
        \exp (\mbox{Re}\;G(z)) = |F(z)| \leq C\exp \pi |y|,\quad |y| \leq 1/2.
    \end{equation*}
    We conclude by one-periodicity that the harmonic function Re$(G(z))$ is bounded by a linear polynomial and hence by Lemma \ref{liouville}, we obtain $G(z) = A + Bz$. The constants $A$ and $B$ can be determined easily, namely $A = \ln \pi $ and $B = 0$, and the result follows.
\end{proof}

The reflection formula is also significant from the point of view of harmonic analysis inasmuch as it appears a very special case of Ramanujan's Master formula \cite[equation (1.30)]{hardy} which was recast by G. H. Hardy as a Paley Wiener theorem for Mellin transforms, that we state below. We also mention the more recent multi-dimensional analogues available in \cite{ding}.

\begin{theorem}[Ramanujan]
    Assume that $\phi$ is a holomorphic function in a half-plane $\mbox{Re}\;z > -\eta\; (\eta > 0)$ and satisfies an estimate of the form
    \begin{equation*}
        |\phi(z)| \le C \exp(q\;\mbox{Re}\;z + r\;|\mbox{Im}\;z|)
    \end{equation*}
    for certain constants $C$, $q$ and $r$ with $0 < r < \pi$. Then,
    \begin{equation}
    \label{eqn:ram}
        \int_0^\infty x^{s-1} \psi(x) dx = \frac{\pi}{\sin \pi s} \phi(-s),
    \end{equation}
    where $\psi(x) = \sum_{n=0}^\infty (-x)^n\phi(n)$ and $0 < \mbox{Re}\;s < \eta$.
\end{theorem}

Note that the series appearing in the expression for $\psi(x)$ converges when $0 < x < \exp(-q)$, but Hardy shows that it continues analytically to a sector containing $[0, \infty)$. Further, note that the integral on the LHS of \eqref{eqn:ram} is the Mellin transform of the function $\psi$. Setting $\phi(x) = 1$ in \eqref{eqn:ram} and using \eqref{eqn:beta}, we obtain the reflection formula for $0 < x < 1$, which can be extended to $\mathbb{C} \setminus \mathbb{Z}$ using analytic continuation.

\section{ Legendre's problems and Landau's theorem}  

Legendre, in his book \cite{legendre}, takes up the problem of determining the values of the gamma function given its values on a subset of $(0, 1]$. The first problem of this kind is finding the least number of values among
\begin{equation}
\label{eqn:legp}
    \Gamma\left(\frac{1}{m}\right), \Gamma\left(\frac{2}{m}\right), \dots, \Gamma\left(\frac{m-1}{m}\right),
\end{equation}
from which all others may be determined by employing \eqref{eqn:fr}, \eqref{eqn:erf} and \eqref{eqn:gmf}. After taking logarithms, the problem translates into a question of computing the rank of a certain matrix with entries $0, 1$ and $-1$. M. A. Stern has shown \cite{stern} that the number of independent numbers in \eqref{eqn:legp} is $\frac{1}{2} \varphi(m)$, where $\varphi$ is the Euler's totient function. An elegant solution to the above problem expressed in terms of a structure theorem for finite abelian groups appears in \cite{nijenhuis}.

The second problem is measure theoretic. Let us say that a subset $S \subseteq (0, \infty)$ is a \textit{fundamental set} for $\Gamma(x)$ if the restriction $\Gamma\vert_{S}$ determines $\Gamma$ completely on $(0, \infty)$ through finitely many applications of \eqref{eqn:fr}, \eqref{eqn:erf} and \eqref{eqn:gmf}. For example, the functional relation \eqref{eqn:fr} immediately gives $(0, 1]$ as a fundamental set. The reflection formula now gives $(0, 1/2) \cup \{1\}$ as a fundamental set. Using the duplication formula and \eqref{eqn:erf}, one can show without much difficulty that the fundamental set $(0, 1/2) \cup \{1\}$ may be shrunk to $(0, 1/3] \cup \{1\}$. In \cite[p. 28]{godefroy}, it is claimed that the fundamental set may be shrunk to $(0, 1/4)$ by virtue of the duplication formula, though it is not so obvious and we provide an elementary argument for the same below. 

Let us introduce some notations before proceeding further. For any subsets $A, B \subseteq (0, \infty)$, we define $A \preceq B$ if the values of the gamma function on $B$ completely determine its values on $A$. Clearly, the relation $\preceq$ is both reflexive and transitive, with the following additional properties: 
\begin{itemize}
    \item[(1)] $A \subseteq B$ implies $A \preceq B$,
    \item[(2)] $A_1 \preceq B_1,\;A_2 \preceq B_2$ implies $A_1 \cup A_2 \preceq B_1 \cup B_2$.
\end{itemize}
\begin{theorem}
    The set $(0, 1/4] \cup \{1/3, 1\}$ is a fundamental set for the gamma function.
\end{theorem}

\begin{proof}
    We proof this in two steps. We first show that $(0, 1/2) \preceq (0, 1/3]$ and then show that $(0, 1/3] \preceq (0, 1/4] \cup \{1/3\}$ to conclude that $(0, 1/4] \cup \{1/3, 1\}$ is a fundamental set for  the gamma function.

    From \eqref{eqn:ldf}, it is evident that for any $z \in (0, \infty)$, we have $\{z/2+1/2\} \preceq \{z, z/2\}$. For $z \in (0, 1/3]$, we have $z/2, z \in (0, 1/3]$ and $z/2 + 1/2 \in (1/2, 2/3]$. This implies $(1/2, 2/3] \preceq (0, 1/3]$. From \eqref{eqn:erf}, we have $[1/3, 1/2) \preceq (1/2, 2/3]$, which further implies $[1/3, 1/2) \preceq (0, 1/3]$. Hence, $(0, 1/2) \preceq (0, 1/3]$.
    
    To show $(0, 1/3] \preceq (0, 1/4] \cup \{1/3\}$, we choose a suitable combination of \eqref{eqn:erf} and \eqref{eqn:ldf} in order to apply a technique similar to the first part of the proof. Setting $z = 2\alpha + 1/2$ and $z = 4\alpha$ in \eqref{eqn:ldf}, we obtain
    \begin{align*}
        \sqrt{\pi} \Gamma(2\alpha + 1/2) &= 2^{2\alpha - 1/2} \Gamma(\alpha + 1/4) \Gamma(\alpha + 3/4),\; \textrm{and}\\
        \sqrt{\pi} \Gamma(4\alpha) &= 2^{4\alpha - 1} \Gamma(2\alpha) \Gamma(2\alpha + 1/2),
    \end{align*}
    respectively. And setting $z = \alpha + 3/4$ in \eqref{eqn:erf}, we obtain
    \begin{equation*}
        \Gamma(\alpha + 3/4) \Gamma (1/4 - \alpha) = \frac{\pi}{\sin \pi (\alpha + 3/4)}.
    \end{equation*}
    Multiplying the three equations obtained above and cancelling like terms (since gamma is non-vanishing), we obtain
    \begin{equation}
    \label{eqn:comb}
        \Gamma(4\alpha) \Gamma (1/4 - \alpha) = \frac{2^{6 \alpha - 3/2}}{\sin \pi (\alpha + 3/4)} \Gamma(2\alpha) \Gamma(\alpha + 1/4).
    \end{equation}
    From the above equation, it is evident that for any $\alpha \in (0, 1/4)$, we have  $\{\alpha + 1/4\} \preceq \{4\alpha, 2\alpha, 1/4 - \alpha\}$. For $\alpha \in (0, 1/4^2)$, we have $4\alpha, 2\alpha, 1/4 - \alpha \in (0, 1/4)$ and $\alpha + 1/4 \in (1/4, 1/4 + 1/4^2)$. This implies $(1/4, 1/4 + 1/4^2) \preceq (0, 1/4)$. Therefore, we obtain $(0, 1/4 + 1/4^2) \preceq (1, 1/4]$. Now, for $\alpha \in (0, 1/4^2 + 1/4^3)$, we have $2\alpha, 4\alpha, 1/4 - \alpha \in (0, 1/4 + 1/4^2)$ and $\alpha + 1/4 \in (1/4, 1/4 + 1/4^2 + 1/4^3)$. This implies $(1/4, 1/4 + 1/4^2 + 1/4^3) \preceq (0, 1/4 + 1/4^2)$. Thus, we obtain $(0, 1/4 + 1/4^2 + 1/4^3) \preceq (0, 1/4]$. Repeating this process infinitely many times, we obtain $(0, 1/3) \preceq (0, 1/4]$. Finally, we have $(0, 1/3] \preceq (0, 1/4] \cup \{1/3\}$.
\end{proof}

\begin{definition}[Germinating function]
    A continuous function $f: (0, \infty) \longrightarrow \mathbb{R}$ is said to be a \textit{germinating function} if there are measurable subsets $S \subseteq (0, \infty)$ of arbitrarily small measure such that $f\vert_S$ determines $f$ completely through finitely many applications of functional identities.
\end{definition}

A. M. Legendre \cite{legendre} posed the problem of finding fundamental subsets having measure as small as possible. E. Landau \cite{landau} provided the solution to this problem by proving that the gamma function is a germinating function. For the benefit of the readers, we include Landau's proof of this result. The following lemma would be needed.

\begin{lemma}[Landau]
\label{landau}
    For any $\delta \in (0,1]$ and an arbitrary interval $(\alpha, \beta]$ with $0 \le \alpha < \beta \le 1$, we can find a non-negative integer $m$ and intervals $I, \{J_i\;:\; 1 \leq i \leq m\}$ 
    such that
    \begin{equation*}
        (\alpha, \beta] \preceq I \cup J_1 \cup \dots \cup J_m,
    \end{equation*}
    where $I \subseteq (0, \delta/2]$ and $J_i \subseteq (\delta/2 , 1]$ for each $1 \leq i \leq m$. Moreover, the lengths of these intervals satisfy the conditions
    \begin{equation*}
        |I| + \sum_{i=1}^m |J_i| = \beta - \alpha, \text{ and } |I| > \frac{\delta}{4} (\beta - \alpha).
    \end{equation*}
\end{lemma}

\begin{proof}
    If $\beta \le \delta/2$, there is nothing to prove since $I = (\alpha, \beta]$ satisfies the given conditions. If $\beta > \delta/2$, let $m$ be the least natural number such that
    \begin{equation*}
        \frac{\beta}{2^m} \le \frac{\delta}{2}.
    \end{equation*}
    Therefore, we have
    \begin{equation*}
        \frac{\beta}{2^{m-1}} > \frac{\delta}{2},
    \end{equation*}
    which further implies
    \begin{equation*}
        \frac{1}{2^m} \ge \frac{\beta}{2^m} > \frac{\delta}{4}.
    \end{equation*}
    From equation \eqref{eqn:ldf}, it is clear that
    \begin{equation*}
        (\alpha, \beta] \preceq \left(\frac{\alpha}{2}, \frac{\beta}{2}\right] \cup \left(\frac{\alpha}{2} + \frac{1}{2}, \frac{\beta}{2} + \frac{1}{2}\right].
    \end{equation*}
    The second interval is a subset of $(1/2, 1] \subseteq (\delta/2, 1]$. Thus, we do not transform it further. Through an $m$-fold application of \eqref{eqn:ldf} on the first interval, we obtain
    \begin{equation*}
        (\alpha, \beta] \preceq \left(\frac{\alpha}{2^m}, \frac{\beta}{2^m}\right] \cup J_1 \cup \dots \cup J_m,
    \end{equation*}
    where $J_i \subseteq (1/2, 1] \subseteq (\delta/2, 1]$ and the sum of the lengths of the intervals $I = (\alpha/2^m, \beta/2^m], \{J_i\}_{i \in [m]}$ is equal to $\beta - \alpha$. Moreover, the interval $I \subseteq (0, \delta/2]$ and its length is
    \begin{equation*}
        |I| = \frac{\beta - \alpha}{2^m} > \frac{\delta}{4}(\beta - \alpha).
    \end{equation*}
\end{proof}

\begin{theorem}[Landau]
\label{thm:landau}
    The gamma function $\Gamma: (0, \infty) \to \mathbb{R}$ is a germinating function.
\end{theorem}

\begin{proof}
    For any $\delta > 0$, we show that there exists a fundamental set for the gamma function having Lebesgue measure less than $\delta$. We begin with the fundamental interval $(\alpha, \beta] = (0, 1]$. From Lemma \ref{landau}, we obtain intervals $I_0, J_1, \dots, J_m$ such that $I_0 \subseteq (0, \delta/2]$ and $J_i \subseteq (\delta/2, 1]$ for each $i \in [m]$. Moreover, these intervals satisfy the conditions
    \begin{equation*}
        (0, 1] \preceq I \cup J_1 \cup \dots \cup J_m, \text{ and } \sum_{i=1}^m |J_i| < (1 - \delta/4).
    \end{equation*}
    Further, applying Lemma \ref{landau} to each of the $J_i$s, we obtain intervals $I_i, J_{i, 1}, \dots, J_{i, m_i}$ such that
    $I_i \subseteq (0, \delta/2]$ and $J_{i, j} \subseteq (\delta/2, 1]$ for each $j \in [m_i]$. Moreover, we have
    \begin{equation*}
        (0, 1] \preceq I \cup \left(\bigcup_{i=1}^m I_i\right) \cup  \left(\bigcup_{i=1}^m \bigcup_{j=1}^{m_i} J_{i, j}\right), \text{ and }\sum_{i=1}^m \sum_{j=1}^{m_i} |J_{i, j}| < (1 - \delta/4)^2.
    \end{equation*}
    After $t$ iterations, we can find a finite collection of intervals $\{I_a\}_{a \in A}$ and $\{J_b\}_{b \in B}$ such that $I_a \subseteq (0, \delta/2]$ for each $a \in A$ and $J_b \subseteq (\delta/2, 1]$. Moreover, these intervals satisfy the condition
    \begin{equation*}
        (0, 1] \preceq \left(\bigcup_{a \in A} I_a\right) \cup \left(\bigcup_{b \in B} J_b\right), \text{ and } \sum_{b \in B} |J_b| < (1 - \delta/4)^t.
    \end{equation*}
    We can choose $t$ so large that
    \begin{equation*}
        (1 - \delta/4)^t < \delta/2,
    \end{equation*}
    or equivalently
    \begin{equation*}
        t < \frac{\log (\delta/2)}{\log (1 - \delta/4)}.
    \end{equation*}
    Thus, we have found a fundamental set for the gamma function having Lebesgue measure less than $\delta$.
\end{proof}

The following result states that the above theorem is sharp and cannot be improved further, thereby settling completely Legendre's original problem of finding fundamental sets having measure as small as possible, for the gamma function.

\begin{theorem}
\label{thm:zero}
    Let $A \subseteq (0, \infty)$ be a set with Lebesgue measure zero. Then, $A$ is not a fundamental set for the gamma function.
\end{theorem}

\begin{proof}
    Let $E$ be the countable collection of identities containing \eqref{eqn:fr}, \eqref{eqn:erf}, and \eqref{eqn:gmf} for each $n \in \mathbb{N}$. For $i \in \mathbb{N}$, let $A_i$ be the set of points on which the value of the gamma function can be determined by at most $i$ applications of identities in $E$ to the set of values of the gamma function on the set $A_0 = A$. Let us choose any identity $e \in E$. Then, $e$ contains finitely many, say $k$, gamma values. For example, the $n=2$ instance of \eqref{eqn:gmf} contains $k=3$ gamma values. Let $T_e$ be the collection of $k(k-1)$ affine transformations which take the argument of one of these $k$ gamma values to that of another. Then, it is easy to check that 
    $$
    A_i \subseteq A_{i-1} \cup \left(\bigcup_{e \in E} \bigcup_{t \in T_e} t(A_{i-1})\right)
    $$ 
    for each $i \in \mathbb{N}$. Therefore, $\lambda(A_{i-1}) = 0$ implies $\lambda(A_i) = 0$ since the Lebesgue measure $\lambda$ is complete. Then, it follows by induction that $\lambda(A_i) = 0$ for each $i \in \mathbb{N} \cup \{0\}$. The set of points on which the values of the gamma function can be determined using finitely many identities in the set $E$ is equal to 
    $B = \bigcup_{i = 0}^{\infty}  A_i$. 
    It follows that $\lambda(B) = \sum_{i=0}^\infty \lambda(A_i) = 0$, and hence $A$ is not a fundamental set for the gamma function.
\end{proof}

Throughout the above discussion, we considered values of gamma function only at positive real points. However, it is possible to extend these results to the full domain of definition $D = \C \setminus \{0, -1, -2, \dots\}$, of the gamma function. The theorem given below is an analogue of Theorem \ref{thm:landau}.

\begin{theorem}
    For any $\delta > 0$, there exists a measurable subset $S \subseteq D$ with Lebesgue measure less than $\delta$ so that $\Gamma|_{S}$ determines $\Gamma$ completely through finitely many applications of functional identities.
\end{theorem}

\begin{proof}
    Let $\delta > 0$. For each $n \in \mathbb{N} \cup \{0\}$, we define the set $S_n$ by
    \begin{equation*}
        S_n = \{x + i y: x \in (0, \infty)\text{ and }-2^{n} < y < 2^{n}\}.
    \end{equation*}
    Further, we define
    \begin{equation*}
        S = \{x + i y: x \in I\text{ and }-1 < y < 1\},
    \end{equation*}
    where $I \subseteq (0, \infty)$ is a measurable fundamental set for the gamma function with $\lambda(I) < \frac{\delta}{2}$, as constructed in the proof of Theorem \ref{thm:landau}. We claim that $S$ is the required set. To show this, we first note that $S$ is Lebesgue measurable and $\lambda(S) < 2 \cdot \frac{\delta}{2} = \delta$. Next, we show that $S_0 \preceq S$ ($\preceq$ can be defined for subsets of $D$ in the same way as before). This can be shown by imitating the proof of Theorem \ref{thm:landau} combined with that fact that
    \begin{equation*}
        \text{Re}\left(\frac{z}{2} + \frac{m}{2}\right) = \frac{\text{Re}(z)}{2} + \frac{m}{2},
    \end{equation*}
    for $m \in \{0, 1\}$, and
    \begin{equation*}
        -1 < \text{Im}\left(\frac{z}{2}\right)= \text{Im}\left(\frac{z}{2}+\frac{1}{2}\right) < 1,
    \end{equation*}
    for $|\text{Im}(z)| < 1$. Now, we show that $S_n \preceq S_{n-1}$ for each $n \in \mathbb{N}$. This follows directly from duplication formula since $z/2, z/2+1/2 \in S_{n-1}$ for $z \in S_n$. Using the transitive property of $\preceq$, we conclude that $S_n \preceq S$ for each $n \in \mathbb{N}$. For any point $z_0$ in the right half-plane $R$, there exists an $n_0 \in \mathbb{N}$ such that $z_0 \in S_{n_0}$. Therefore, we have $R \preceq S$. Further, using \eqref{eqn:erf} (or by repeated application of \eqref{eqn:fr}), we obtain $D \preceq R$, which when combined with $R \preceq S$ gives $D \preceq S$.
\end{proof}

Finally, we remark that there does not exist any set $S \subseteq D$ having Lebesgue measure zero such that $\Gamma|_{S}$ determines $\Gamma$ completely on its domain through finitely many applications of \eqref{eqn:fr}, \eqref{eqn:erf} and \eqref{eqn:gmf}. The proof of this statement is exactly the same as that of Theorem \ref{thm:zero}.

However, there still remains the question of the existence of minimal fundamental sets leading to the following interesting problem.

\begin{problem}[Minimal fundamental sets]
    Do minimal fundamental sets for the gamma function exist? If so, does there exist a measurable/non-measurable minimal fundamental set?
\end{problem}

\section{Schl\"omilch's generalization of the duplication formula}

While the refection formula and the duplication formula are multiplicative in nature, there is a beautiful additive formula due to O. Schl\"omilch \cite[Section 5]{schlomilch} which generalizes the duplication formula. The result of Schl\"omilch seems to have passed into oblivion since there is no mention of it in modern works on special functions. In this section, we further generalize the formula of Schl\"omilch. The original proof of Schl\"omilch employs certain remarkable transformations of integrals but as such cannot be adapted to prove our proposed generalization. We have thus included both the original proof of Schl\"omich as well as a new proof of the generalized version. 

\begin{theorem}[Schl\"omilch]
    \label{thm:schlomilch}
    Let $m$ be a non-negative integer and $z \in \C$ be such that $\text{Re}(z) > m$. Then,
    \begin{equation}
    \label{eqn:sch}
        \frac{2^{z-1}}{\sqrt{\pi}}\Gamma\left(\frac{z+m+1}{2}\right)\Gamma\left(\frac{z-m}{2}\right) = \sum_{n = 0}^{m} \frac{\Gamma(z-n)}{2^n n!} (m-n+1)_{2n},
    \end{equation}
    where $(.)_{2n}$ is the Pochhammer symbol.
\end{theorem}

\begin{proof}
    We begin with  the following well-known identity for $\cos ku$ \cite[p. 180]{bromwich}, where $k = 2m + 1$, is an odd natural number and $u \in \mathbb{C}$.
    \begin{equation}
    \label{eqn:cosine}
        \cos ku = \cos u \sum_{n=0}^{m} \frac{(-1)^n}{(2n)!} \sin^{2n}u \left\{\prod_{j=1}^{n} (k^2 - (2j-1)^2)\right\}.
    \end{equation}
    Setting $x = e^{iu}$ and using
    \begin{equation*}
        \cos ku = \frac{1}{2} \left( x^k + \frac{1}{x^k} \right),\;
        \cos u  = \frac{1}{2} \left( x + \frac{1}{x} \right),\;\text{and}\;\sin u  = \frac{i}{2} \left( x - \frac{1}{x} \right),
    \end{equation*}
    in \eqref{eqn:cosine}, we obtain
    \begin{equation*}
        x^k + \frac{1}{x^k} = \left( x + \frac{1}{x} \right) \sum_{n=0}^{m} \frac{1}{2^{2n}(2n)!} \left(x - \frac{1}{x}\right)^{2n} \left\{\prod_{j=1}^{n} (k^2 - (2j-1)^2)\right\}.
    \end{equation*}
    Dividing both sides by $x$ and using $k = 2m + 1$, we obtain
    \begin{equation}
    \label{eqn:sch-1}
        x^{2m} + \frac{1}{x^{2m+2}} = \left(1 + \frac{1}{x^2}\right) \sum_{n=0}^m \frac{1}{(2n)!} \left(x - \frac{1}{x}\right)^{2n} (m - n + 1)_{2n}.
    \end{equation}
    For $n = 0, 1, \dots, m$, let us define
    \begin{equation}
    \label{eqn:M}
        M_{2n} = \frac{1}{(2n)!} (m - n + 1)_{2n},
    \end{equation}
    so that \eqref{eqn:sch-1} can be written as
    \begin{equation*}
        x^{2m} + \frac{1}{x^{2m+2}} = \left( 1 + \frac{1}{x^2} \right) \sum_{n=0}^m M_{2n} \left(x - \frac{1}{x}\right)^{2n}.
    \end{equation*}
    Let $z$ be a complex number with $\text{Re}(z) > m$. Dividing the above equation by $\left(x^2 + \frac{1}{x^2}\right)^{z+\frac{1}{2}}$ and integrating from $0$ to $1$, we obtain
    \begin{equation}
    \label{eqn:int}
        \int_0^1 \frac{x^{2m} dx}{\left(x^2 + \frac{1}{x^2}\right)^{z+\frac{1}{2}}} + \int_0^1 \frac{dx}{x^{2m + 2} \left(x^2 + \frac{1}{x^2}\right)^{z+\frac{1}{2}}} = \sum_{n=0}^m M_{2n} \int_0^1 \frac{\left(x - \frac{1}{x}\right)^{2n}}{\left(x^2 + \frac{1}{x^2}\right)^{z+\frac{1}{2}}} \left(1 + \frac{1}{x^2}\right) dx.
    \end{equation}
    Substituting $y = \frac{1}{x}$ in the second integral appearing in the above equation, the left hand side equals
    \begin{align*}
        \int_0^\infty \frac{y^{2m + 2z + 1} dy}{\left(y^4 + 1\right)^{z+\frac{1}{2}}}.
    \end{align*}
    Further, substituting $r = y^4$ in the above expression, we recognize it as a beta integral and appealing to the beta-gamma relation \eqref{eqn:bg}, we get
    \begin{equation}
    \label{eqn:lhs}
        \textrm{LHS of }\eqref{eqn:int} = \frac{1}{4} \frac{\Gamma \left( \frac{z + m + 1}{2} \right) \Gamma \left( \frac{z - m}{2} \right)}{\Gamma \left( z + \frac{1}{2} \right)}.
    \end{equation}
    The integrals appearing on the right hand side  of \eqref{eqn:int} are likewise beta integrals as is evident from the discussion that follows.
    \begin{equation*}
        \int_0^1 \frac{\left(x - \frac{1}{x}\right)^{2n}}{\left(x^2 + \frac{1}{x^2}\right)^{z+\frac{1}{2}}} \left(1 + \frac{1}{x^2}\right) dx = \int_0^1 \frac{\left(\frac{1}{x} - x\right)^{2n}}{\left(2 + \left(\frac{1}{x} - x\right)^2 \right)^{z+\frac{1}{2}}} \left(1 + \frac{1}{x^2}\right) dx.
    \end{equation*}
    Upon substituting $z = \frac{1}{x} - x$ followed by $r = \frac{z^2}{2}$, the above integral takes the form
    \begin{equation*}
        \frac{2^n}{2^{z+1}} \int_0^\infty \frac{r^{n - \frac{1}{2}} dr}{(1 + r)^{z+\frac{1}{2}}},
    \end{equation*}
    which, as asserted is a beta integral in view of \eqref{eqn:beta}. Appealing to the beta-gamma relation \eqref{eqn:bg}, we obtain
    \begin{equation}
    \label{eqn:rhs}
        \frac{2^n}{2^{z+1}} \int_0^\infty \frac{r^{n - \frac{1}{2}} dr}{(1 + r)^{z+\frac{1}{2}}} = \frac{2^n}{2^{z+1}} \frac{\Gamma\left(n + \frac{1}{2}\right) \Gamma(z - n)}{\Gamma\left(z + \frac{1}{2}\right)}.
    \end{equation}
    Using \eqref{eqn:lhs} and \eqref{eqn:rhs} in \eqref{eqn:int}, and multiplying both sides by $\Gamma\left(z + \frac{1}{2}\right)$, we obtain
    \begin{equation}
        \frac{1}{4} \Gamma \left( \frac{z + m + 1}{2} \right) \Gamma \left( \frac{z - m}{2} \right) = \frac{1}{2^{z+1}} \sum_{n=0}^m 2^n M_{2n} \Gamma\left(n + \frac{1}{2}\right) \Gamma(z - n).
    \end{equation}
    Further substituting the values of $M_{2n}$ from \eqref{eqn:M} and using the known values
    \begin{equation*}
        \Gamma\left(n + \frac{1}{2}\right) = \frac{(2n)!}{2^{2n}n!}\sqrt{\pi},
    \end{equation*}
    of the gamma function, we obtain the desired result.
\end{proof}

Observe that when $m = 0$, the above theorem reduces to the duplication formula of Legendre. Further, setting $z = m + 2l + 1$ for some non-negative integer $l$, in \eqref{eqn:sch} gives
\begin{equation*}
    \binom{m + l}{m} = \sum_{n=0}^m \frac{1}{2^{m + n}} \binom{m + n}{m} \binom{2l + m - n}{2l},
\end{equation*}
which bears some resemblance with the Chu-Vandermonde identity.

We now state and prove a generalization of the above theorem. The proof involves expressing the infinite series in the right hand side of \eqref{eqn:sch} as a hypergeometric function and application of some hypergeometric identities and transformations along with Euler's reflection formula.

\begin{theorem}
    Let $w, z \in \mathbb{C}$ be such that $w+z-1/2$ is non-integer and $z, w$ are not non-positive integers. Then,
    \begin{equation}
        \frac{1}{\sqrt{2\pi}} \frac{2^{w+z} \Gamma(w) \Gamma(z)}{1 - \cot{w\pi} \cot{z\pi}} = \sum_{n = 0}^{\infty} \frac{\Gamma(w+z-n-1/2)}{2^n n!} (w-z-n+1/2)_{2n},
    \end{equation}
    where $(.)_{2n}$ is the Pochammer symbol. 
\end{theorem}

\begin{proof}
    We begin with the RHS and show that it is equal to the LHS. Multiplying and dividing the RHS by $\Gamma(w + z - 1/2)$, and expressing the Pochammer symbol as a gamma quotient, we obtain
    \begin{equation}
    \label{eqn:sum}
        \sum_{n = 0}^{\infty} \frac{\Gamma(w+z-n-1/2)}{2^n n!} (w-z-n+1/2)_{2n} = \Gamma(w+z-1/2) \sum_{n = 0}^{\infty} \frac{\Gamma(w+z-n-1/2)}{\Gamma(w+z-1/2)} \frac{\Gamma(w-z+n+1/2)}{\Gamma(w-z-n+1/2)} \frac{(1/2)^n}{n!}.
    \end{equation}
    The coefficient in the summation on the RHS of the above equation can further be written as
    \begin{equation}
    \label{eqn:coeff}
         \frac{\Gamma(w+z-n-1/2)}{\Gamma(w+z-1/2)} \frac{\Gamma(w-z+n+1/2)}{\Gamma(w-z-n+1/2)} = (w-z+1/2)_n \frac{\Gamma(w+z-n-1/2)}{\Gamma(w+z-1/2)} \frac{\Gamma(w-z+1/2)}{\Gamma(w-z-n+1/2)}.
    \end{equation}
    Now using Euler's reflection formula (or the functional relation \eqref{eqn:fr}), we have
    \begin{equation*}
        \frac{\Gamma(w-z+1/2)}{\Gamma(w-z-n+1/2)} = (-1)^n \frac{\Gamma(z-w+n+1/2)}{\Gamma(z-w+1/2)},
    \end{equation*}
    and
    \begin{equation*}
        \frac{\Gamma(w+z-n-1/2)}{\Gamma(w+z-1/2)} = (-1)^n \frac{\Gamma(3/2-w-z)}{\Gamma(3/2-w-z+n)}.
    \end{equation*}
    Using the above two equations in \eqref{eqn:coeff} and expressing gamma quotients as Pochammer symbols, we obtain
    \begin{equation*}
         \frac{\Gamma(w+z-n-1/2)}{\Gamma(w+z-1/2)} \frac{\Gamma(w-z+n+1/2)}{\Gamma(w-z-n+1/2)} = \frac{(w-z+1/2)_n (z-w+1/2)_n}{(3/2-w-z)_n}.
    \end{equation*}
    Using the above equation in \eqref{eqn:sum}, we get
    \begin{equation*}
        \sum_{n = 0}^{\infty} \frac{\Gamma(w+z-n-1/2)}{2^n n!} (w-z-n+1/2)_{2n} = \Gamma(w+z-1/2) \sum_{n = 0}^{\infty} \frac{(w-z+1/2)_n (z-w+1/2)_n}{(3/2-w-z)_n} \frac{(1/2)^n}{n!}.
    \end{equation*}
    The infinite sum on the RHS of the above equation can be identified as a hypergeometric function to obtain
    \begin{equation*}
        \sum_{n = 0}^{\infty} \frac{\Gamma(w+z-n-1/2)}{2^n n!} (w-z-n+1/2)_{2n} = \Gamma(w+z-1/2) \hphantom{|}_2 F_1 (w-z+1/2,z-w+1/2;3/2-w-z;1/2).
    \end{equation*}
    Applying Euler's transformation \cite[Theorem 2.2.5]{andrews} to the hypergeometric function in the above equation, we get
    \begin{equation*}
        \sum_{n = 0}^{\infty} \frac{\Gamma(w+z-n-1/2)}{2^n n!} (w-z-n+1/2)_{2n} = \Gamma(w+z-1/2) (1-1/2)^{1/2-w-z} \hphantom{|}_2 F_1 (1-2w,1-2z;3/2-w-z;1/2).
    \end{equation*}
    The parameters of the hypergeometric function ${}_2 F_1 (a,b;c;1/2)$ in the above equation obey the condition $c = \frac{1}{2}(a + b + 1)$, and hence using Gauss's second summation theorem \cite[Theorem 3.5.4(i)]{andrews}, we obtain
    \begin{equation*}
        \sum_{n = 0}^{\infty} \frac{\Gamma(w+z-n-1/2)}{2^n n!} (w-z-n+1/2)_{2n} = 2^{w+z-1/2} \Gamma(w+z-1/2) \frac{\Gamma(1/2) \Gamma(3/2-w-z)}{\Gamma(1-w)\Gamma(1-z)},
    \end{equation*}
    which can further be written as
    \begin{equation*}
        \sum_{n = 0}^{\infty} \frac{\Gamma(w+z-n-1/2)}{2^n n!} (w-z-n+1/2)_{2n} = 2^{w+z-1/2} \Gamma(1/2) \frac{\Gamma(w) \Gamma(z) \Gamma(w+z-1/2)\Gamma(3/2-w-z)}{\Gamma(w) \Gamma(1-w) \Gamma(z) \Gamma(1-z)}.
    \end{equation*}
    Now, applying Euler's reflection formula to the gamma products $\Gamma(s)\Gamma(1-s)$ for $s = w, z,\text{ and }w+z-1/2$, the above equation reduces to
    \begin{align*}
        \sum_{n = 0}^{\infty} \frac{\Gamma(w+z-n-1/2)}{2^n n!} (w-z-n+1/2)_{2n} &= 2^{w+z-1/2} \Gamma(1/2) \frac{\Gamma(w) \Gamma(z) \sin \pi w \sin \pi z}{\pi \sin \pi (w+z-1/2)},\\
        &= - 2^{w+z-1/2} \Gamma(1/2) \frac{\Gamma(w) \Gamma(z) \sin \pi w \sin \pi z}{\pi \cos \pi (w+z)}.
    \end{align*}
    Using the known value $\Gamma(1/2) = \sqrt{\pi}$ and the angle-sum formula for cosine in the above equation yields the desired result.
\end{proof}

We finally remark that replacing $w$ with $\frac{z+m+1}{2}$ and $z$ with $\frac{z-m}{2}$ in the above theorem yields Theorem \ref{thm:schlomilch}.

\end{document}